\documentclass[11pt]{amsart}

\usepackage{verbatim}
\usepackage{eucal,url,amssymb,stmaryrd,
enumerate,amscd,
}

\usepackage[hypertex]{hyperref}      

\usepackage{amsfonts}
\usepackage{amsmath,amsthm,amssymb,amscd,enumerate,eucal,url,stmaryrd}

\numberwithin{equation}{section}

\newtheorem{thrm}{Theorem}[section]

\newtheorem{lemma}[thrm]{Lemma}

\newtheorem{prop}[thrm]{Proposition}

\newtheorem{cor}[thrm]{Corollary}

\newtheorem{rmrk}[thrm]{Remark}

\begin{document}

\begin{abstract}
We construct explicit compact supersymmetric solutions with non-zero
field strength, non-flat instanton and constant dilaton to the
heterotic string equations in dimension five. We present a quadratic
condition on the curvature which is necessary and sufficient the
heterotic supersymmetry and the anomaly cancellation to imply the
heterotic equations of motion in dimension five. We supply compact
nilmanifold in dimension 5 satisfying the heterotic supersymmetry
equations with non-zero fluxes and constant dilaton which obeys  the
three-form Bianchi identity and solves the heterotic equations of
motion in dimension five.
\end{abstract}

\title[ Heterotic equations of motion  with non-zero fluxes and
constant dilaton] {Compact supersymmetric solutions of the heterotic
equations of motion  in dimension 5}
\date{\today}

\author{Marisa Fern\'andez}
\address[Fern\'andez]{Universidad del Pa\'{\i}s Vasco\\
Facultad de Ciencia y Tecnolog\'{\i}a, Departamento de Matem\'aticas\\
Apartado 644, 48080 Bilbao\\ Spain} \email{marisa.fernandez@ehu.es}

\author{Stefan Ivanov}
\address[Ivanov]{University of Sofia ``St. Kl. Ohridski"\\
Faculty of Mathematics and Informatics\\
Blvd. James Bourchier 5\\
1164 Sofia, Bulgaria} \email{ivanovsp@fmi.uni-sofia.bg}

\author{Luis Ugarte}
\address[Ugarte]{Departamento de Matem\'aticas\,-\,I.U.M.A.\\
Universidad de Zaragoza\\
Campus Plaza San Francisco\\
50009 Zaragoza, Spain} \email{ugarte@unizar.es}

\author{Raquel Villacampa}
\address[Villacampa]{Departamento de Matem\'aticas\,-\,I.U.M.A.\\
Universidad de Zaragoza\\
Campus Plaza San Francisco\\
50009 Zaragoza, Spain} \email{raquelvg@unizar.es}

\maketitle

\setcounter{tocdepth}{2} \tableofcontents

\section{Introduction. Field and Killing-spinor equations}

The bosonic fields of the ten-dimensional supergravity which arises
as low energy effective theory of the heterotic string are the
spacetime metric $g$, the NS three-form field strength $H$, the
dilaton $\phi$ and the gauge connection $A$ with curvature $F^A$.
The bosonic geometry considered in this paper is of the form
$R^{1,9-d}\times M^d$ where the bosonic fields are non-trivial only
on $M^d$, $d\leq 8$. We consider the two connections
\begin{equation}\label{pmcon}
\nabla^{\pm}=\nabla^g \pm \frac12 H,
\end{equation}
where $\nabla^g$ is the Levi-Civita connection of the Riemannian
metric $g$. Both connections preserve the metric, $\nabla^{\pm}g=0$ and
have totally skew-symmetric torsion $\pm H$, respectively.


The Green-Schwarz anomaly cancellation mechanism requires that the
three-form Bianchi identity receives an $\alpha'$ correction of the
form
\begin{equation}\label{acgen}
dH=\frac{\alpha'}48\pi^2(p_1(M^d)-p_1(E))=\frac{\alpha'}4
\Big(Tr(R\wedge R)-Tr(F^A\wedge F^A)\Big),
\end{equation}
where $p_1(M^d), p_1(E)$ are the first Pontrjagin forms of $M^d$
with respect to a connection $\nabla$ with curvature $R$ and the
vector bundle $E$ with connection $A$, respectively.

A class of heterotic-string backgrounds for which the Bianchi
identity of the three-form $H$ receives a correction of type
\eqref{acgen} are those with (2,0) world-volume supersymmetry. Such
models were considered in \cite{HuW}. The target-space geometry of
(2,0)-supersymmetric sigma models has been extensively investigated
in \cite{HuW,Str,HP1}. Recently, there is revived interest in these
models \cite{GKMW,CCDLMZ,GMPW,GMW,GPap} as string backgrounds and in
connection to heterotic-string compactifications with fluxes
\cite{Car1,BBDG,BBE,BBDP,y1,y2,y3,y4}.

In writing \eqref{acgen} there is a subtlety to the choice of
connection $\nabla$ on $M^d$ since anomalies can be cancelled
independently of the choice \cite{Hull}. Different connections
correspond to different regularization schemes in the
two-dimensional worldsheet non-linear sigma model. Hence the
background fields given for the particular choice of $\nabla$ must
be related to those for a different choice by a field redefinition
\cite{Sen}. Connections on $M^d$ proposed to investigate the
anomaly cancellation  \eqref{acgen} are $\nabla^g$ \cite{Str,GMW},
$\nabla^+$ \cite{CCDLMZ}, $\nabla^-$ \cite{Berg,Car1,GPap,II},
Chern connection $\nabla^c$ when $d=6$ \cite{Str,y1,y2,y3,y4}.

A heterotic geometry will preserve supersymmetry if and only if, in
10 dimensions, there exists at least one Majorana-Weyl spinor
$\epsilon$ such that the supersymmetry variations of the fermionic
fields vanish, i.e. the following Killing-spinor equations hold
\cite{Str}
\begin{gather} \nonumber
\delta_{\lambda}=\nabla_m\epsilon = \left(\nabla_m^g
+\frac{1}{4}H_{mnp}\Gamma^{np} \right)\epsilon=\nabla^+\epsilon=0;
\\\label{sup1} \delta_{\Psi}=\left(\Gamma^m\partial_m\phi
-\frac{1}{12}H_{mnp}\Gamma^{mnp} \right)\epsilon=(d\phi-\frac12H)\cdot\epsilon=0; \\
\nonumber
\delta_{\xi}=F^A_{mn}\Gamma^{mn}\epsilon=F^A\cdot\epsilon=0,
\end{gather}
where  $\lambda, \Psi, \xi$ are the gravitino, the dilatino and the
gaugino  fields, respectively and $\cdot$ means Clifford action of
forms on spinors.

The bosonic part of the ten-dimensional supergravity action in the
string frame is \cite{Berg}
\begin{gather}\label{action}
S=\frac{1}{2k^2}\int
d^{10}x\sqrt{-g}e^{-2\phi}\Big[Scal^g+4(\nabla^g\phi)^2-\frac{1}{2}|H|^2
-\frac{\alpha'}4\Big(Tr |F^A|^2)-Tr |R|^2\Big)\Big].
\end{gather}

The string frame field equations (the equations of motion induced
from the action \eqref{action}) of the heterotic string up to
two-loops \cite{HT} in sigma model perturbation theory are (we use
the notations in \cite{GPap})
\begin{gather}\nonumber
Ric^g_{ij}-\frac14H_{imn}H_j^{mn}+2\nabla^g_i\nabla^g_j\phi-\frac{\alpha'}4
\Big[(F^A)_{imns}(F^A)_j^{mns}-R_{imns}R_j^{mns}\Big]=0;\\\label{mot}
\nabla^g_i(e^{-2\phi}H^i_{jk})=0;\\\nonumber
\nabla^+_i(e^{-2\phi}(F^A)^i_j)=0,
\end{gather}
The field equation of the dilaton $\phi$ is implied from the first
two equations above.

We search for  solutions to lowest nontrivial order in $\alpha'$
of the equations of motion that follow from the bosonic action
which also preserves at least one supersymmetry.

It is known \cite{Bwit,GMPW} (\cite{GPap} for dimension $d=6$),
that the equations of motion of type I supergravity \eqref{mot}
with $R=0$ are automatically satisfied if one imposes, in addition
to the preserving supersymmetry equations \eqref{sup1}, the
three-form Bianchi identity \eqref{acgen} taken with respect to a
flat connection on $TM, R=0$.

A lot of effort had been done in dimension $d=6$ and compact
torsional solutions for the heterotic/type I string are known to
exist \cite{DRS,BBDG,BBE,CCDLMZ,GMW,y1,y2,y3,y4,DFG,FIUV}. In
dimensions $d=7$ and $d=8$ non-compact heterotic/type I  solutions
with non-zero fluxes to the equations of motion preserving at
least one supersymmetry are  constructed in
\cite{FNu,FN,HS,GNic,II} and the first compact torsional solutions
are presented recently in \cite{FIUV1}.

In dimension $d=5$, to the best of our knowledge, there are not
known any compact solution either to the supersymmetry equations
\eqref{sup1} or to the heterotic equations of motion \eqref{mot}
with non-zero fluxes. If the field strength vanishes, $H=0$, the
5-dimensional case reduces to dimension four since any five
dimensional Riemannian spin manifold admitting $\nabla^g$-parallel
spinor is reducible. Non compact solutions on circle bundle over
4-dimensional base endowed with a hyper K\"ahler metric (when the
4-dimensional metric is Egushi-Hanson, Taub-NUT, Atiyah-Hitchin)
have appeared in \cite{LV-P,GGMPR,SM,BBW,Pap}, the compact cases
are discussed  in \cite{GMW} where a cohomological obstruction is
presented.

The main goal of this paper is to construct explicit compact
supersymmetric valid solutions with non-zero field strength,
non-flat instanton and constant dilaton to the heterotic equations
of motion \eqref{mot} in dimension $d=5$.

It was known \cite{FI,FI2} that solutions to the first two Killing
spinor equations in dimension $d=5$ are quasi-Sasaki manifolds
with anti-self-dual exterior derivative of the almost contact form
and their special conformal transformations (see the precise
definitions below). In particular, Sasakian manifolds can not
solve the heterotic supersymmetry equations. In the case when the
quasi-Sasaki structure is regular the solutions to the first two
equations in \eqref{sup1} are $S^1$-bundles over a Calabi-Yau
4-manifold with anti-self-dual curvature 2-form. { This fact
was generalized recently in \cite{Pap} for supersymmetric
solutions of the heterotic string with holonomy of $\nabla^+$
contained in $SU(2)$, $Hol(\nabla^+)\subseteq SU(2)$. The
explicit compact five-dimensional} solutions we present in this paper are $S^1$-bundles
over a 4-torus.

In Theorem~\ref{shypo}, Theorem~\ref{shypo1} we give structure
equations of any solution to the first two Killing spinor equations
in \eqref{sup1}  in terms of exterior derivatives of an
$SU(2)$-structure in dimension five, a notion introduced in
\cite{ConS,GGMPR}, and express its Ricci tensor in terms of the
structure forms. Based on the analysis made in \cite{FI,FI2} we
also simplify the formula for the torsion tensor of the unique almost contact
metric connection with totally skew-symmetric torsion described in \cite{FI} { thus obtaining simple
formula for the NS three-form field strength}.

According to no-go (vanishing) theorems  (a consequence of the
equations of motion \cite{FGW,Bwit}; a consequence of the
supersymmetry \cite{IP1,IP2} for SU($n$)-case and \cite{GMW} for the
general case) there are no compact solutions with non-zero flux and
non-constant dilaton satisfying simultaneously the supersymmetry
equations \eqref{sup1} and the three-form  Bianchi identity
\eqref{acgen} if one takes flat connection on $TM$, more precisely a
connection with zero first Pontrjagin 4-form, $Tr(R\wedge R)=0$.
Therefore, in the compact case one necessarily has to have a
non-zero term $Tr(R\wedge R)$. However, under the presence of a
non-zero curvature 4-form $Tr(R\wedge R)$ the solution of the
supersymmetry equations \eqref{sup1} and the anomaly cancellation
condition \eqref{acgen} obeys the second and the third equations of
motion but does not always satisfy the Einstein equation of motion
(the first equation in \eqref{mot}). We give in Theorem \ref{thac} a
quadratic expression for $R$ which is necessary and sufficient
condition in order that \eqref{sup1} and \eqref{acgen} imply
\eqref{mot} in dimension $d=5$ based on the properties of the special
geometric structure induced from the first two equations in
\eqref{sup1}. (A similar condition in dimension six, seven and eight
we presented in \cite{FIUV,FIUV1}, respectively). In particular, if
$R$ is an $SU(2)$-instanton the supersymmetry equations together with
the anomaly cancellation condition imply the equations of motion in dimension 5. The latter
can also be seen following the considerations in the Appendix of
\cite{GMPW}.

We present in Theorem \ref{N(2,1)} compact nilmanifolds $N(2,1)_{a,b,c}$
in dimension five satisfying the heterotic supersymmetry equations
\eqref{sup1} with non-zero flux $H$, non-trivial instanton and
constant dilaton obeying the three-form Bianchi identity
\eqref{acgen} with curvature term $R=R^+$ which also solve the
heterotic equations of motion \eqref{mot}.  Our solutions {  depend on six real constants,} are
$S^1$-bundles over a 4-torus and  seem to be the first explicit
compact valid supersymmetric heterotic solutions with non-zero
flux and constant dilaton in dimension 5 satisfying the equations
of motion \eqref{sup1}. { These solutions can be viewed as
examples of  half-symmetric solutions, i.e. heterotic solutions
with 8 supersymmetries preserved. Our explicit solutions
also provide new compact half-symmetric solutions in the spirit of \cite{Pap},
(see also \cite{GLP,GPRS}) for which $Hol(\nabla^+)\subseteq
SU(2)$ and the group acting on the 10-dimensional space is
$T^{5,1}$ with the geometry splitting as $T^{4,1}\times N(2,1)_{a,b,c}$. }

Finally, in Proposition \ref{N(2,1)-LC} we show
 compact nilmanifolds in  dimension five
satisfying the heterotic supersymmetry equations \eqref{sup1} with non-zero fluxes and  the
three-form Bianchi identity \eqref{acgen} with curvature term
$R=R^g$.

\begin{rmrk}
We do not know compact non-regular solutions to the first two heterotic
Killing spinor equations, i.e. compact quasi-Sasaki manifolds of this kind
with non-closed orbits of the Reeb vector field, or, in view of
Theorem \ref{shypo} below,  compact 5-manifolds satisfying the
structure equations \eqref{solstr} with non-closed orbits of the Reeb vector field.
\end{rmrk}

{\bf Our conventions:} We  rise and lower the indices with the
metric and use the summation convention on repeated indices. For
example, $$B_{ijk}C^{ijk}=B_i^{jk}C^i_{jk}=B_{ijk}C_{ijk}=
\sum_{ijk=1}^nB_{ijk}C_{ijk}.$$

The connection 1-forms $\omega_{ji}$ of a metric  connection
$\nabla, \nabla g=0$ with respect to a local basis
$\{E_1,\ldots,E_n\}$ are given by
$$
\omega_{ji}(E_k) = g(\nabla_{E_k}E_j,E_i),
$$
since we write $\nabla_X E_j = \omega^s_j(X)\, E_s$.

The curvature 2-forms $\Omega^i_j$ of  $\nabla$ are given in terms
of the connection 1-forms $\omega^i_j$ by
$$
\Omega^i_j = d \omega^i_j + \omega^i_k\wedge\omega^k_j, \quad
 \Omega_{ji} = d \omega_{ji} +
\omega_{ki}\wedge\omega_{jk}, \quad R^l_{ijk}=\Omega^l_k(E_i,E_j),
\quad R_{ijkl}=R^s_{ijk}g_{ls}.
$$
and the first Pontrjagin class is represented by the 4-form
$$
p_1(\nabla)={1\over 8\pi^2} \sum_{1\leq i<j\leq d}
\Omega^i_j\wedge\Omega^i_j.
$$

\section{Geometry of the heterotic supersymmetry  equations}

Geometrically, the vanishing of the gravitino variation is
equivalent to the existence of a non-trivial real spinor parallel
with respect to the metric connection $\nabla^+$ with totally
skew- symmetric torsion $T=H$. The presence of $\nabla^+$-parallel
spinor leads to restriction of the holonomy group $Hol(\nabla^+)$
of the torsion connection $\nabla^+$.  A detailed analysis of the
induced geometries is carried out in \cite{GMW} and all possible
geometries (including non compact stabilizers) are investigated in
\cite{GLP,GPRS,GPR,Pap,P}.

The existence of $\nabla^+$-parallel spinor in dimension 5
determines an almost contact metric structure whose properties as
well as solutions to gravitino and dilatino Killing-spinor equations
are investigated in \cite{FI,FI2}.

We recall that an almost contact metric structure consists of an odd
dimensional manifold $M^{2k+1}$ equipped with a Riemannian metric
$g$, vector field $\xi$ of length one, its dual 1-form $\eta$ as
well as an endomorphism $\psi$ of the tangent bundle such that
\begin{equation}\label{acont1}
\psi(\xi)=0, \quad \psi^2=-id +\eta\otimes\xi, \quad
g(\psi.,\psi.)=g(.,.)-\eta\otimes\eta.
\end{equation}
In local coordinates \eqref{acont1} reads
\begin{equation*}
\psi^i_j\xi^j=0, \quad \psi^i_s\psi^s_j=-\delta^i_j +\eta_j\xi^i, \quad
g_{st}\psi^s_i\psi^t_j=g_{ij}-\eta_i\eta_j.
\end{equation*}
The Reeb vector field $\xi$ is determined by the equations
$\eta(\xi)=\eta_s\xi^s=1,\quad (\xi\lrcorner d\eta)_i=d\eta_{si}\xi^s=0$,
where $\lrcorner$ denotes the interior multiplication.
The fundamental form $F$ is defined by $$F(.,.)=g(.,\psi.), \quad F_{ij}=g_{is}\psi^s_j.$$
The Nijenhuis tensor $N$   of an almost
contact metric structure is given by
\begin{equation*}
 N=[\psi.,\psi.]+\psi^2[.,.] -\psi
[\psi.,.]-\psi[.,\psi.]  +d\eta\otimes\xi.
\end{equation*}
There are many special types of almost contact metric structures. We
introduce those which are relevant to our considerations:
\begin{enumerate}
\item[-] normal almost contact structures determined by the condition $N=0$;
\item[-] contact metric structures characterized by $d\eta=2F$;
\item[-] quasi-Sasaki structures, $N=0, dF=0$.
Consequently, $\xi$ is a Killing vector \cite{Bl};
\item[-] Sasaki structures, $N=0, d\eta=2F$.
Consequently, $\xi$ is a Killing vector \cite{Bl}.
\item[-]  the class of almost contact metric structures with totally skew-symmetric
Nijenhuis tensor and Killing vector field  $\xi$ introduced in
\cite{FI}.
\end{enumerate}

In dimension five any solution to the gravitino Killing spinor
equation, i.e. any parallel spinor with respect to a metric
connection with torsion 3-form defines an almost contact metric
structure $(g,\xi,\eta,\psi)$ via the formulas
$$\xi\cdot\epsilon=  \sqrt{-1} \cdot\epsilon ,
\qquad -2\psi X\cdot\epsilon +\xi\cdot X\cdot\epsilon = \sqrt{-1}X\cdot\epsilon,
$$
which is preserved by the torsion connection.

An almost contact metric structure admits a linear connection
$\nabla^+$ with torsion 3-form preserving the structure, i.e.
$\nabla^+ g=\nabla^+\xi=\nabla^+\psi=0$, if and only if the Nijenhuis
tensor is totally skew-symmetric, $g(N(X,Y),Z)=-g(N(X,Z),Y)$, and the vector field $\xi$ is a
Killing vector field \cite{FI}. In this case the torsion
connection is unique. The torsion $T$ of $\nabla^+$  is expressed by (\cite{FI}, Theorem 8.2)
\begin{equation}\label{toracon}
T=\eta\wedge d\eta+d^{\psi}F+N-\eta\wedge (\xi\lrcorner N),
\end{equation}
where $d^{\psi}F=-dF(\psi.,\psi.,\psi), \quad (d^{\psi}F)_{ijk}=-dF_{str}\psi^s_i\psi^t_j\psi^r_k$.

In particular one has
$ d\eta_{ij}=(\xi\lrcorner T)_{ij}=T_{sij}\xi^s, \quad (\xi\lrcorner d\eta)_i=T_{sti}\xi^s\xi^t=0$.

In fact, \eqref{toracon} simplifies since if the Nijenhuis tensor is totally skew-symmetric
then $\xi$ is a Killing vector field  exactly when (\cite{FI2}, Proposition 3.1 and its proof)
\begin{equation}\label{simpl}
(\xi\lrcorner dF)_{ij}=dF_{sij}\xi^s=0 \Leftrightarrow (\xi\lrcorner N)_{ij}=N_{sij}\xi^s=0.
\end{equation}
The proof of Lemma 8.3 in \cite{FI} also yields
$d\eta(.,.)=d\eta(\psi.,\psi.)$. In local coordinates the latter reads $d\eta_{ij}=d\eta_{st}\psi^s_i\psi^t_j$.

Now, Theorem 8.2 in \cite{FI} is formulated as follows
\begin{thrm}\cite{FI,FI2}
An almost contact metric structure admits a unique linear connection
$\nabla^+$ with torsion 3-form preserving the structure, i.e.
$\nabla^+ g=\nabla^+\xi=\nabla^+\psi=0$, if and only if the Nijenhuis
tensor is totally skew-symmetric and  $\xi\lrcorner N=0$.  The torsion $T$ of $\nabla^+$  is expressed by
\begin{equation}\label{toraconn}
T=\eta\wedge d\eta+d^{\psi}F+N,
\end{equation}
which expresses in local coordinates as
$T_{ijk}=(\eta\wedge d\eta)_{ijk}-dF_{str}\psi^s_i\psi^t_j\psi^r_k+N_{ijk}.$
\end{thrm}
Since $\nabla^+\xi=0$ the restricted holonomy group $Hol(\nabla^+)$ of $\nabla^+$ containes in $U(k)$.
The spinor bundle $\Sigma$ of a $(2k+1)$-dimensional almost contact metric spin
manifold decomposes under the action of the fundamental 2-form
$F$ into the sum $\Sigma=\Sigma^0\oplus \dots \oplus\Sigma^k,$ $dim(\Sigma^r)=\Big({k\atop r}\Big)$.
The isotropy group of a spinor of type $\Sigma^0$ or $\Sigma^k$ coincides with  the subgroup
$SU(k)\subset U(k)$. Consequently, there exists locally a $\nabla^+$-parallel spinor of
type $\Sigma^0$ or $\Sigma^k$ exactly when $Hol(\nabla^+)\subset SU(k)$.
The equivalent curvature  condition  found in Proposition 9.1 \cite{FI} reads
\begin{equation}\label{su2}
R^+_{ijkl}F^{kl}=0 \Leftrightarrow Ric^+_{ij}=-\nabla^+_i\theta_j-\frac14\psi^s_jdT_{islm}F^{lm},
\end{equation}
where the Lee form $\theta$ is defined in \cite{FI2} by
\begin{equation}\label{thet5}
\theta_i=\frac12\psi^s_iT_{skl}F^{kl}=\frac12dF_{ikl}F^{kl}.
\end{equation}
Consequently, $\theta(\xi)=0$.

(Warning: note that the Lee form $\omega^{\nabla}$ defined in \cite{FI}
differs slightly from $\theta, \omega^{\nabla}(\psi.)=\theta(.)$).

It is shown in \cite{FI2} that solutions to the both gravitino and
dilatino Killing spinor equations are connected with a special type
conformal transformations of an almost contact metric structure
introduced in \cite{FI2} by
\begin{equation}\label{conf}
\psi':=\psi, \quad \eta':=\eta, \quad \xi':=\xi, \quad g':=e^{2f}g+(1-e^{2f})\eta\otimes\eta,
\end{equation}
where $f$ is a smooth function which is  constant along the integral curves of $\xi, df(\xi)=0$.
The new torsion $T'$ and Lee form $\theta'$ are given by
$$T'=T+(e^{2f}-1)d^{\psi}F+2e^{2f}d^{\psi}f\wedge F, \qquad \theta'=\theta+2df,
$$
where $d^{\psi}f=-df(\psi.)$.  In local coordinates we have  $(d^{\psi}f)_i=-df_s\psi^s_i.$

We restrict our attention to dimension five.
In  dimension five the Nijenhuis tensor is totally skew-symmetric exactly when it
vanishes \cite{CM}, i.e. the structure is normal.  In this case $\xi$ is
automatically a Killing vector field \cite{Bl}, the Lee form determines completely
the three form $dF$ due to \eqref{simpl}, $dF=\theta\wedge F$ and the dilatino
equation (the second equation in \eqref{sup1}) admits a solution exactly when
the normal almost contact manifold is special conformal to a quasi-Sasaki
5-manifold (\cite{FI2}, Theorem 5.5). If  the non-trivial spinor
$\epsilon\in \Sigma^1$ then the space is special conformal to the
standard Sasaki structure on the 5-dimensional Heisenberg group.

For a non-trivial
spinor $\epsilon\in \{\Sigma^0,\Sigma^2\}$ the dilatino equation
admits a solution if and only if  the next equalities hold (\cite{FI2}, Proposition 5.5)
\begin{equation}\label{dilz}
2d\phi=\theta, \quad  *_{\mathbb H}d\eta=-d\eta,
\end{equation}
where  $*_{\mathbb H}$ denote the Hodge operator acting  in the
4-dimensional orthogonal complement $\mathbb H$ of the vector $\xi$,
$\mathbb H=Ker \eta$. We call an $\mathbb H$-valued 2-form
satisfying the second equation of \eqref{dilz}  $\mathbb
H$-anti-self dual.

In this case the torsion (the NS three-form $H$) is given by
\begin{equation}\label{tsol5}
H=T=\eta\wedge d\eta +2d^{\psi}\phi\wedge F
\end{equation}
and the space is special conformal to a quasi-Sasaki 5-manifold with
$\mathbb H$-anti-self-dual 2 form $d\eta$.   In particular, there is
no solution on any Sasaki 5-manifold.

\begin{rmrk}\label{remim} The $\nabla^+$-parallel spinor  $\epsilon\in
\{\Sigma^0,\Sigma^2\}$ transforms under special conformal
transformations into a $(\nabla^+)'$-parallel spinor since $d\eta$
is $\mathbb H$-anti-self-dual [\cite{FI2}, Theorem 5.1].
Hence, a solution to the supersymmetry equations \eqref{sup1} in
dimension five reduces to solve \eqref{sup1} in the case of constant
dilaton.  Then any special conformal transformation gives again  a
solution to \eqref{sup1} and the anomaly cancellation condition
could be reduced to a highly non-linear PDE for a real function $f$.
\end{rmrk}

The simplest case is when  the normal almost contact structure is
regular, i.e. the orbit space $N^4=M^5/\xi$ is a smooth manifold.
Then $M^5$ is a principal $S^1$-bundle with $\mathbb
H$-anti-self-dual curvature form equal to $d\eta$ and any spinor
$\epsilon\in \{\Sigma^0,\Sigma^2\}$ solving the gravitino and
dilatino Killing spinor equations projects to $N^4$. Indeed, its Lie
derivative $\mathbb L_{\xi}\epsilon$ (see \cite{BG}), calculated in
\cite{FI2} is
$$\mathbb L_{\xi}\epsilon=\nabla^g_{\xi}\epsilon-\frac14d\eta\cdot\epsilon=
-\frac12d\eta\cdot\epsilon=0$$
since $d\eta$ is anti-self-dual and $\epsilon\in \{\Sigma^0,\Sigma^2\}$.
Now, Theorem 3.2 in \cite{FI2}  and \eqref{conf} yield

\begin{thrm}\cite{FI2}\label{dil} If  $(M^5,g,\eta,\xi,\psi)$ is a compact
regular almost contact metric manifold solving the gravitino and
dilatino Killing spinor equations for $\epsilon\in
\{\Sigma^0,\Sigma^2\}$ then $M^5$ is an $S^1$-bundle over a flat
torus or a K3-surface  with an $\mathbb H$-anti-self-dual curvature
equal to $d\eta$. The metric has the form
\begin{gather}\label{cys1}
g_5=e^{2f}g_{cy}+\eta\otimes\eta,
\end{gather}
where $g_{cy}$ is the Calabi-Yau metric on the 4-dimensional base
and $f$ is a smooth function on it.

The dilaton $\phi=2f$ depends only on the Calabi-Yau 4-manifold.

The torsion is $T=\eta\wedge d\eta-2e^{2f}(df\circ\psi)\wedge F$ and
the flux $H=T$.
\end{thrm}

\begin{rmrk} The third case of Theorem 3.2 in \cite{FI2}, namely $S^1$-bundle over
the Hopf surfaces, should be excluded since this solves the dilatino equation only locally.
\end{rmrk}

\subsection{The $SU(2)$-structure point of view.} The gravitino Killing
spinor equation, i.e. the $\nabla^+$-parallel spinor
$\epsilon\in\{\Sigma^0,\Sigma^2\}$ defines a reduction of the
structure group $SO(5)$ to $SU(2)$ which is described in terms of
forms by Conti and Salamon in \cite{ConS} (see also \cite{GGMPR}) as
follows: an $SU(2)$-structure on 5-dimensional manifold $M$ is
$(\eta,F=F_1,F_2,F_3)$, where $\eta$ is a $1$-form dual to $\xi$ via
the metric and $F_p, p=1,2,3$ are $2$-forms on $M$ satisfying
  \begin{equation}\label{defsu2}
  F_p\wedge F_q=\delta_{pq}v, \quad
  v\wedge\eta\not=0,
  \end{equation}
for some $4$-form $v$, and
  \begin{equation}\label{defsu2-2}
  X\lrcorner F_1=Y\lrcorner F_2\Rightarrow F_3(X,Y)\ge
  0.
  \end{equation}
Now $\mathbb H=Ker\eta$ and the 2-forms $F_p, p=1,2,3$ can be chosen
to form a basis of the $\mathbb H$-self-dual 2-forms \cite{ConS}.
The $SU(2)$-structure $(\eta,F=F_1,F_2,F_3)$ is $\nabla^+$-parallel,
$\nabla^+\eta=\nabla^+F_p=0, p=1,2,3$, since the defining spinor
$\epsilon$ is $\nabla^+$-parallel.

Involving the dilatino equation, the second equation in
\eqref{sup1}, we get that $d\eta$ is an $\mathbb H$-anti-self-dual
2-form. We  show that if the dilaton is constant then the 2-forms
$F_p, p=1,2,3$ are harmonic, i.e. closed and co-closed. We have
\begin{thrm}\label{shypo}
The first two equations in \eqref{sup1} admit a solution with
constant dilaton in dimension five exactly when there exists a five
dimensional manifold $M$ endowed with an $SU(2)$-structure
$(M,\eta,F=F_1,F_2,F_3)$ satisfying the structure equations:
\begin{equation}\label{solstr}
dF_p=0, \qquad *_{\mathbb H}d\eta = - d\eta.
\end{equation}
The Ricci tensors are given by
\begin{equation}\label{ricc5}
Ric^+_{mn}=-d\eta_{im}d\eta_{in}, \qquad
Ric^g_{mn}=-\frac12d\eta_{im}d\eta_{in}+\frac14|d\eta|^2\eta_m\eta_n.
\end{equation}

In particular, if $M$ is compact then the second and the third Betti
numbers $b_2(M),b_3(M)$ are greater than or equal to three, $b_2(M)\ge 3,
b_3(M)\ge 3$.
\end{thrm}
\begin{proof}
Suppose $(M,\eta,F=F_1,F_2,F_3)$ is a solution to the first two
equations in \eqref{sup1} with constant dilaton. Then the structure
is quasi-Sasaki, $N=dF_1=0$, the torsion of $\nabla^+$ is given by
$T=\eta\wedge d\eta$ and $d\eta$ is $\mathbb H$-anti-self-dual by
the discussion in the previous subsection. We shall show that the
$\mathbb H$-self-dual forms $F_p$ are closed and therefore harmonic.
Let $(e_1,e_2,e_3,e_4,e_5=\xi)$ be an orthonormal basis. Then we
calculate for $X,Y,Z\in TM$ that
\begin{multline}\label{lcivf}
(\nabla^g_XF_p)(Y,Z)=(\nabla^+_XF_p)(Y,Z)
-\frac12\sum_{s=1}^5\Big[T(X,Y,e_s)F_p(e_s,Z)+T(X,Z,e_s)F_p(Y,e_s)\Big]
\\=-\frac12\sum_{s=1}^4\Big[(\eta\wedge
d\eta)(X,Y,e_s)F_p(e_s,Z)+(\eta\wedge d\eta)(X,Z,e_s)F_p(Y,e_s)\Big]
\end{multline}
since $\nabla^+F_p=0$ and $\xi\lrcorner F_p=0$. The equation
\eqref{lcivf} yields
\begin{itemize}
\item Let $X,Y,Z\in \mathbb H$. Then $\nabla^gF_p|_{\mathbb H}=0$.
Consequently, $dF_p|_{\mathbb H}=0$.
\item Let $X=\xi, Y,Z\in\mathbb H$. Then
\begin{equation}\label{ans}
(\nabla^g_{\xi}F_p)(Y,Z)=\frac12\sum_{s=1}^4\Big[F_p(Z,e_s)d\eta(Y,e_s)-F_p(Y,e_s)d\eta(Z,e_s)\Big]=0.
\end{equation}
The last equality is a pure algebraic consequence of the fact that
$F_p$ is $\mathbb H$-self-dual while $d\eta$ is $\mathbb
H$-anti-self-dual.
\item Let $X,Y\in\mathbb H,Z=\xi$. Then
\begin{equation}\label{asn1}
(\nabla^g_XF_p)(Y,\xi)=\frac12\sum_{s=1}^4F_p(Y,e_s)d\eta(X,e_s).
\end{equation}
\end{itemize}
Using \eqref{ans} and \eqref{asn1} we get
$dF_p(\xi,Y,Z)=2(\nabla^g_{\xi}F_p)(Y,Z)=0$, for $Y,Z\in\mathbb H$.
Hence, $dF_p=0$.

For the converse, we consider the Riemannian product
$N=M\times\mathbb R$ with the $SU(3)$-structure
$(\Omega,\Psi=\Psi^++\sqrt{-1}\Psi^-)$ with K\"ahler form $\Omega$
and complex volume form $\Psi$ defined by
\begin{gather}\label{su3}
\Omega=-F_1-\eta\wedge dt;\qquad \Psi^+=F_2\wedge\eta - F_3\wedge
dt;\qquad \Psi^-=F_3\wedge\eta + F_2\wedge dt.
\end{gather}
Using \eqref{solstr} we easily derive from \eqref{su3} that
$d(\Omega\wedge\Omega)=d\Psi^+=d\Psi^-=0$, i.e. it is a balanced
hermitian structure with holomorphic complex volume form. In
particular, the almost contact metric structure on $M$ is normal.
Applying Theorem 4.1 and Corollary 4.3 from \cite{II} we conclude
\begin{equation}\label{su31}
\nabla^+_N\Omega=\nabla^+_N\Psi^+=\nabla^+_N\Psi^-=0,
\end{equation}
where $\nabla^+_N$ is the Bismut-Strominger metric connection with
torsion 3-form $T_N$ given by
\begin{equation}\label{su3t}
T_N=-*_Nd\Omega=\eta\wedge d\eta
\end{equation}
where $*_N$ denotes the Hodge operator on $N$ and we have used the
first equation in \eqref{su3} and the fact that $d\eta$ is $\mathbb
H$-anti-self-dual.

Hence, the torsion $T_N$ does not depend on $\mathbb R$ and
therefore the connection $\nabla^+_N$ descends to $M$. Now,
\eqref{su31} yield $\nabla^+_N\eta=\nabla^+_NF_p=0, p=1,2,3$ and the
descended connection on $M$ coincides with $\nabla^+$ as two metric
connections with equal torsion must coincide.

The formulas for the Ricci tensors \eqref{ricc5} follow just taking
$df=0$ into \eqref{ricc5f} established below.

The last assertion follows from the fact that the three 2-forms
$F_p$ and the three 3-forms $F_p\wedge\eta$ are harmonic and
represent different cohomology classes. Indeed, the equation
$*_{\mathbb H}F_p=F_p$ implies $\delta
F_p=-*d*F_p=-*d(F_p\wedge\eta)=-*(dF_p\wedge\eta + F_p\wedge
d\eta)=0$ since $d\eta$ is $\mathbb H$-anti-self-dual.
\end{proof}

Combine Remark~\ref{remim} with Theorem~\ref{shypo} using
\eqref{conf} to derive
\begin{thrm}\label{shypo1}
The first two equations in \eqref{sup1} admit a solution  in
dimension five exactly when there exists a five dimensional manifold
$M$ endowed with an $SU(2)$-structure $(M,\eta,F=F_1,F_2,F_3)$
satisfying the structure equations:
\begin{equation}\label{solstr1}
dF_p=2df\wedge F_p, \qquad *_{\mathbb H}d\eta = - d\eta, \qquad
df(\xi)=0.
\end{equation}
The flux $H$ is given by
\begin{equation}\label{tor5f}
H=T=\eta\wedge d\eta +2d^{\psi}f\wedge F,
\end{equation}
where $\psi$ is the almost complex structure on $\mathbb H$ defined
by $g(X,\psi Y)=F(X,Y)$.

The dilaton $\phi$ is equal to $\phi=2f.$

The Ricci tensors are given by:
\begin{gather}\label{ricc5f}
Ric^+_{mn}=-d\eta_{im}d\eta_{in}+\nabla^g_mdf_n+\Delta_{\mathbb
H}fg_{mn}+df_i(d\eta_{im}\eta_n-d\eta_{in}\eta_m), \\\nonumber
Ric^g_{mn}=-\frac12d\eta_{im}d\eta_{in}+\frac14|d\eta|^2\eta_m\eta_n+
\nabla^g_mdf_n+\Delta_{\mathbb H}fg_{mn}-2df_mdf_n+2|df|^2g_{mn},
\end{gather}
where $\Delta_{\mathbb H}f:=\sum_{i=1}^4(\nabla^g_{e_i}df)e_i$ is
the horizontal subLaplacian and $|df|^2:=\sum_{i=1}^4df(e_i)^2$ is
the norm of the horizontal gradient.

In particular, if $M$ is compact then the second and the third Betti
numbers $b_2(M),b_3(M)$ are greater than or equal to three, $b_2(M)\ge 3,
b_3(M)\ge 3$.
\end{thrm}
\begin{proof}
We need to prove only \eqref{ricc5f}. We shall use \eqref{ric+ff}
and \eqref{mo} from below. Using \eqref{tor5f} and \eqref{solstr1},
we calculate
\begin{gather}-\frac14\psi^s_ndT_{msij}F^{ij}=-4d\eta_{im}d\eta_{in}+2(dd^{\psi}f)_{ij}F^{ij}-
8|df|^2g_{mn};\nonumber\\\nonumber
2(dd^{\psi}f)_{ij}F^{ij}=-4\nabla^+_idf_i+2df_i\psi^j_iT_{jst}F^{st}=-4\Delta_{\mathbb
H}f+8|df|^2;\\\label{callc}
df_iT_{imn}=df_i(d\eta_{im}\eta_n-d\eta_{in}\eta_m);\\
T_{mij}T_{nij}=2d\eta_{mi}d\eta_{ni}+|d\eta|^2\eta_m\eta_n-8df_mdf_n+8|df|^2g_{mn}.\nonumber
\end{gather}
We get the first equality in \eqref{ricc5f} from \eqref{ric+ff} and
the first three equalities in \eqref{callc}. The second equality
follows from the already proved first one, the fourth equality in
\eqref{callc} and \eqref{mo} below.
\end{proof}

We note that another proof of the second equality in \eqref{ricc5} and
\eqref{ricc5f} can be derived from the general formula of the Ricci
tensor for a general $SU(2)$-structure presented in \cite{BV}.

In addition to gravitino and dilatino Killing spinor equations, the
vanishing of the gaugino variation requires the 2-form $F^A$ to be
of instanton type (\cite{CDev,Str,HS,RC,DT,GMW}). In dimension five,
an $SU(2)$-instanton i.e the gauge field $A$ is a connection with
curvature 2-form $F^A\in su(2)$. The $SU(2)$-instanton condition can
be written in the form \cite{CDev,Str}
\begin{equation}\label{5inst}
F^A_{mn}=-\frac{1}{2}F^A_{st}(F\wedge F)^{st}\hspace{0mm}_{mn}.
\end{equation}

In this paper we consider compact regular $SU(2)$-manifolds in
dimension 5, more precisely the case of $S^1$-bundles over a flat
4-torus. We find compact solutions to \eqref{sup1} satisfying the
anomaly cancellation \eqref{acgen} with non-zero fluxes, constant
dilaton which also solves the heterotic equations of motion
\eqref{mot}.

\begin{rmrk}
It seems to be of particular interest whether there are compact non-regular (the
integral curves of the Reeb vector field $\xi$ are not closed)  quasi-Sasaki 5-manifolds
with anti-self-dual 2-form $d\eta$ whose Riemannian Ricci tensor is
given by \eqref{ricc5}, or equivalently,  non-regular $SU(2)$-
structures obeying \eqref{solstr} on compact 5-manifold. We do not
know any examples of this kind. Such examples might be relevant in
constructing compact heterotic solutions  in dimension six since the
construction of $\mathbb T^2$-bundles over Calabi-Yau surface
presented in \cite{GP} can be generalized to a circle bundle over
such an example solving automatically the first two equations in
\eqref{sup1}(see \cite{FIUV1}).
\end{rmrk}

\section{Heterotic supersymmetry and equations of motion}

It is known \cite{Bwit,GMPW} (\cite{GPap} for dimension 6), that
the equations of motion of type I supergravity \eqref{mot} with
$R=0$ are automatically satisfied if one imposes, in addition to
the preserving supersymmetry equations \eqref{sup1}, the
three-form Bianchi identity \eqref{acgen} taken with respect to a
flat connection on $TM, R=0$. However, the no-go theorems
\cite{FGW,Bwit,IP1,IP2,GMW} state that if even $Tr R\wedge R=0$
there are no compact solutions with non-zero flux $H$ and
non-constant dilaton.

In the presence of a curvature term $Tr R\wedge R\not=0$ a solution
of the supersymmetry equations \eqref{sup1} and the anomaly
cancellation condition \eqref{acgen} obeys the second and the third
equations in \eqref{mot} but does not always satisfy the Einstein
equation of motion (the first equation in \eqref{mot}). However if
the curvature $R$ is of instanton type \eqref{sup1} and
\eqref{acgen} imply \eqref{mot} which can also be seen follow the
considerations in the Appendix of \cite{GMPW}. We have

\begin{thrm}\label{thac}
The Einstein equation of motion (the first equation in \eqref{mot})
is a consequence of the heterotic Killing spinor equations
\eqref{sup1} and  the anomaly cancellation \eqref{acgen} if and only
if the next identity holds
\begin{equation}\label{supmot}
\frac1{2}\Big[R_{msij}R_{trij}+R_{mtij}R_{rsij}+R_{mrij}R_{stij}\Big]F^{tr}\psi^s_n
=R_{mstr}R_n^{str}.
\end{equation}
In particular, if $R$ is an instanton then  \eqref{supmot} holds.
\end{thrm}
\begin{proof}

The Ricci tensors are connected by (see e.g. \cite{FI})
\begin{gather}\label{ricg+}
Ric^g_{mn}=Ric^+_{mn}+\frac14T_{mst}T_n^{st}-\frac12\nabla^+_sT^s_{mn},
\qquad
Ric^+_{mn}-Ric^+_{nm}=\nabla^+_sT^s_{mn}=\nabla^g_sT^s_{mn},\\\label{mo}
Ric^g_{mn}=\frac12(Ric^+_{mn}+Ric^+_{nm})+\frac14T_{mst}T_n^{st}.
\end{gather}
In view of \eqref{dilz}, the Ricci tensor described in \eqref{su2} is given by

\begin{equation}\label{ric+ff}
Ric^+_{mn}=-2\nabla^+_md\phi_n-\frac14\psi^s_ndT_{msij}F^{ij}=
-2\nabla^g_md\phi_n+d\phi_sT^s_{mn}-\frac14\psi^s_ndT_{msij}F^{ij}.
\end{equation}
Substitute \eqref{ric+ff} into \eqref{mo}, insert the result into
the first equation of \eqref{mot} and use the anomaly cancellation
\eqref{acgen} to conclude the assertion.
\end{proof}
It is shown in \cite{I2} that the curvature of $R^+$ satisfies the
identity $R^+_{ijkl}=R^+_{klij}$ if and only if
$\nabla^+_iT_{jkl}$ is a four form. Now Theorem \ref{thac} yields

\begin{cor}\label{thacp}
Suppose the torsion 3-form is $\nabla^+$-parallel, $\nabla^+_iT_{jkl}=0$. The  equations of motion
\eqref{mot} with respect to  the curvature $R^+$ of the (+)-connection are consequences of the
heterotic Killing spinor equations \eqref{sup1} and  the anomaly
cancellation \eqref{acgen}.
\end{cor}

\subsection{Heterotic supersymmetric equations of motion with constant
dilaton}\label{cdil}

In the case when the dilaton is constant we arrive to the
following problems:

We look for a compact 5-manifold $M$ with an $SU(2)$-structure
$(\eta,F_p), p=1,2,3$ which satisfies the following conditions
\begin{enumerate}
\item[a).] Gravitino and dilatino Killing spinor equations
(the first two equations in \eqref{sup1}): the forms $F_p$ are closed
and $d\eta$ is $\mathbb H$-anti-self-dual.
\item[b).] Gaugino Killing spinor equation (the third equation in \eqref{sup1}).
Look for a  vector bundle $E$ of rank $r$ over $M$ equipped with an
$SU(2)$-instanton, i.e. a connection $A$ with curvature 2-form
$\Omega^A$ satisfying \eqref{5inst} written in the form
\begin{equation}\label{25}
(\Omega^A)^m_n(\psi E_k,\psi
E_l)=(\Omega^A)^m_n(E_k,E_l),\qquad
\sum_{k=1}^5(\Omega^A)^m_n(E_k,\psi E_k)=0,
\end{equation}
where $\{E_1,\ldots,E_5=\xi\}$ is an orthonormal basis on $M$ and
$\psi$ is the almost complex structure on $\mathbb H$ defined by
$g(X,\psi Y)=F_1(X,Y), \quad \psi\xi=0$.
\item[c).] Anomaly cancellation condition:

\begin{equation}\label{ac5}
dH=dT=d\eta\wedge
d\eta=\frac{\alpha'}48\pi^2\Big(p_1(M)-p_1(A)\Big), \qquad
\alpha'>0.
\end{equation}
\item[d).] The first Pontrjagin form $p_1(M)$ satisfies equation
\eqref{supmot}.
\end{enumerate}

\section{Explicit compact solutions}

In this section we give an explicit family of compact solutions to
the heterotic supersymmetric equations of motion with constant
dilaton, based on a quotient of the 5-dimensional generalized
Heisenberg group $H(2,1)$.

First, let us recall that $H(2,1)$ is the nilpotent Lie group
consisting of the matrices of the form
$$
H(2,1)=\left\{ \left( \begin{array}{cccc}
1 & x_1 & x_2 & z \\
0 & 1 & 0 & y_1 \\
0 & 0 & 1 & y_2 \\
0 & 0 & 0 & 1
\end{array} \right)
\mid x_i,y_i, z \in \mathbb{R}, 1\leq i\leq 2 \right\}.
$$
For each triple $(a,b,c)\in \mathbb{R}^3$ such that $a^2+b^2\not=0$,
we consider the basis of left invariant 1-forms $e^1,\ldots,e^5$ on
$H(2,1)$ given by
$$
\begin{array}{rll}
\!\!\!&\!\!\! e^1= (a^2+b^2+c^2)\, d x_1,\quad & e^2= a\, d y_1 +
b\left( 1+{c^2\over a^2+b^2} \right) d x_2 - a c\, d y_2,\\[10pt]
\!\!\!&\!\!\! e^3= b\, d y_1 - a \left( 1+{c^2\over a^2+b^2} \right)
d x_2 -b c\, d y_2, \quad & e^4= c\, d y_1 +(a^2+b^2)\, dy_2,\\[12pt]
\!\!\!&\!\!\! e^5 = (a^2+b^2+c^2)^2 (x_1 dy_1 + x_2 dy_2 - dz) . &
\end{array}
$$
In terms of this basis the structure equations of the Lie algebra
${\mathfrak h}(2,1)$ of $H(2,1)$ become
\begin{equation}\label{diff-heisenberg}
\left\{
\begin{array}{rcl}
d e^1 \!\!&\!\!=\!\!&\!\! d e^2 = d e^3 = d e^4 =0,\\[7pt]
d e^5 \!\!&\!\!=\!\!&\!\! a (e^{12} - e^{34})+ b (e^{13} + e^{24}) +
c(e^{14} - e^{23}).
\end{array}
\right.
\end{equation}
Notice that these equations also correspond to ${\mathfrak h}(2,1)$
when $a=b=0$ and $c\not=0$, so (\ref{diff-heisenberg}) are valid for
any triple $(a,b,c)\in \mathbb{R}^3-\{(0,0,0)\}$. It is immediate to
check that the $SU(2)$-structure $(\eta,F_1,F_2,F_3)$ given by
\begin{equation}\label{structure}
F_1= e^{12} + e^{34},\quad\quad F_2= e^{13} + e^{42},\quad\quad F_3=e^{14} + e^{23}, \quad\quad \eta=e^5.
\end{equation}
satisfies~(\ref{solstr}). In view of Theorem~\ref{shypo}, this
family provides explicit solutions (with constant dilaton) to the
first two equations in \eqref{sup1}.

Next we prove that this is the unique family of left invariant
solutions (with constant dilaton) to the first two equations in
\eqref{sup1} on a 5-dimensional Lie group.

\begin{thrm}\label{left-inv-solutions}
Let $\mathfrak g$ be a Lie algebra of dimension $5$ with an
$SU(2)$-structure $(\eta,F_1,F_2,F_3)$ satisfying
$$dF_1=0, \qquad dF_2=0, \qquad dF_3=0, \qquad
*_{\mathbb H}d\eta = - d\eta\not= 0.
$$
Then, $\mathfrak g$ is isomorphic to the Lie algebra ${\mathfrak
h}(2,1)$. Moreover, there is a basis $e^1,\ldots,e^5$ for
${\mathfrak g}^*$ satisfying $(\ref{diff-heisenberg})$ for some
$a,b,c\in \mathbb{R}$ with $a^2+b^2+c^2\not=0$ and such that the
$SU(2)$-structure $(\eta,F_1,F_2,F_3)$ expresses as
$(\ref{structure})$.
\end{thrm}

\begin{proof}
Let us consider a basis $e^1,\ldots,e^5$ for ${\mathfrak g}^*$ such
that the $SU(2)$-structure $(\eta,F_1,F_2,F_3)$ expresses
as~(\ref{structure}). In terms of $e^1,\ldots,e^5$ the equations of
the Lie algebra $\mathfrak g$ are of the form
\begin{equation}\label{gen-dif}
\left\{
\begin{array}{rcl}
d e^i \!\!&\!\!=\!\!&\!\! a_{i1}\, F_1 + a_{i2}\, F_2 + a_{i3}\, F_3
+ b_{i1}\, F^-_1 + b_{i2}\, F^-_2 + b_{i3}\, F^-_3\\[5pt]
&&
+ (c_{i1}\, e^1 + c_{i2}\, e^2 + c_{i3}\, e^3 + c_{i4}\, e^4)e^5,\\[7pt]
d e^5 \!\!&\!\!=\!\!&\!\! a (e^{12} - e^{34})+ b (e^{13} + e^{24}) +
c(e^{14} - e^{23}),
\end{array}
\right.
\end{equation}
for $i=1,\ldots,4$, where $a_{ij},b_{ij},c_{ij}\in \mathbb{R}$ and
$$
F^-_1 = e^{12} - e^{34},\qquad F^-_2 = e^{13} + e^{24}, \qquad F^-_3
= e^{14} - e^{23}.
$$
Let us denote by $F_i^{jkl}$ the component in $e^{jkl}$ of the
3-form $dF_i$. It is easy to see that
$$F_1^{125}=-c_{11}-c_{22},\quad F_2^{135}=-c_{11}-c_{33},\quad F_3^{145}=-c_{11}-c_{44},\quad
F_1^{345}=-c_{33}-c_{44},$$ which imply the vanishing of $c_{ii}$,
for $i=1,\ldots,4$. Moreover,
$$
\begin{array}{llllllll}
& F_1^{135}=-c_{23}+c_{41},\quad & F_1^{245}=c_{14}-c_{32},\quad
& F_2^{125}=-c_{32}-c_{41},\quad & F_2^{345}=c_{14}+c_{23}, \\[5pt]
& F_1^{145}=-c_{24}-c_{31},\quad & F_1^{235}=c_{13}+c_{42},\quad
& F_3^{125}=c_{31}-c_{42},\quad & F_3^{345}=-c_{13}+c_{24},\\[5pt]
& F_2^{145}=c_{21}-c_{34},\quad & F_2^{235}=-c_{12}+c_{43},\quad &
F_3^{135}=-c_{21}-c_{43},\quad & F_3^{245}=-c_{12}-c_{34},
\end{array}
$$
which imply the following equalities:
\begin{equation}\label{ces}
c_{41}=c_{23}=-c_{32}=-c_{14},\qquad
c_{42}=c_{31}=-c_{24}=-c_{13},\qquad c_{43}=-c_{34}=-c_{21}=c_{12}.
\end{equation}

Let $E_1,\ldots,E_5$ be the basis of $\mathfrak g$ dual to
$e^1,\ldots,e^5$, and let us denote by $P_{ijk}^l$ the component in
$E_l$ of $\bigl[[E_i,E_j],E_k\bigr]+ \bigl[[E_j,E_k],E_i\bigr]+
\bigl[[E_k,E_i],E_j\bigr]$, i.e.
$$
\bigl[[E_i,E_j],E_k\bigr]+ \bigl[[E_j,E_k],E_i\bigr]+
\bigl[[E_k,E_i],E_j\bigr]=\sum_{l=1}^5 P_{ijk}^l\, E_l
$$
It is clear that the Jacobi identity of the Lie algebra $\mathfrak
g$ is equivalent to $P_{ijk}^l=0$ for $1\leq i<j<k\leq 5$ and $1\leq
l\leq 5$. From the vanishing of $c_{ii}$ and (\ref{ces}), a direct
calculation shows that
$$P^5_{235}= -2 (b\, c_{12} - a\, c_{13}),\quad P^5_{245}= -2 (c\, c_{12} - a\, c_{14}),\quad
P^5_{345}= -2 (c\, c_{13} - b\, c_{14}).$$ Therefore, $P^5_{235}=P^5_{245}= P^5_{345}=0$
if and only if there is $\lambda\in
\mathbb{R}$ such that
\begin{equation}\label{lambda}
c_{12}=\lambda\, a,\qquad c_{13}=\lambda\, b,\qquad c_{14}=\lambda\,
c.
\end{equation}
Moreover,
\begin{equation}\label{jac-1}
\left\{
\begin{array}{rll}
\!\!\!&\!\!\! P^1_{125} + P^1_{345}= 2\lambda (a\, a_{21} + b\,
a_{31} + c\, a_{41}),\quad & P^2_{125} + P^2_{345}= -2\lambda (a\,
a_{11} + c\, a_{31} - b\,
a_{41}),\\[9pt]
\!\!\!&\!\!\! P^3_{125} + P^3_{345}= -2\lambda (b\, a_{11} - c\,
a_{21} + a\, a_{41} ),\quad & P^4_{125} + P^4_{345}= -2\lambda (c\,
a_{11} + b\, a_{21} - a\, a_{31}),
\end{array}
\right.
\end{equation}
\begin{equation}\label{jac-2}
\left\{
\begin{array}{rll}
\!\!\!&\!\!\! P^1_{135} - P^1_{245}= 2\lambda (a\, a_{22} + b\,
a_{32} + c\, a_{42}),\quad & P^2_{135} - P^2_{245}= -2\lambda (a\,
a_{12} + c\, a_{32} - b\,
a_{42}),\\[9pt]
\!\!\!&\!\!\! P^3_{135} - P^3_{245}= -2\lambda (b\, a_{12} - c\,
a_{22} + a\, a_{42} ),\quad & P^4_{135} - P^4_{245}= -2\lambda (c\,
a_{12} + b\, a_{22} - a\, a_{32}),
\end{array}
\right.
\end{equation}
\begin{equation}\label{jac-3}
\left\{
\begin{array}{rll}
\!\!\!&\!\!\! P^1_{145} + P^1_{235}= 2\lambda (a\, a_{23} + b\,
a_{33} + c\, a_{43}),\ \ & P^2_{145} + P^2_{235}= -2\lambda (a\,
a_{13} + c\, a_{33} - b\,
a_{43}),\\[9pt]
\!\!\!&\!\!\! P^3_{145} + P^3_{235}= -2\lambda (b\, a_{13} - c\,
a_{23} + a\, a_{43} ),\ \ & P^4_{145} + P^4_{235}= -2\lambda (c\,
a_{13} + b\, a_{23} - a\, a_{33}),
\end{array}
\right.
\end{equation}

If $\lambda\not=0$, since $a^2+b^2+c^2\not=0$, then it follows from
(\ref{jac-1}), (\ref{jac-2}) and (\ref{jac-3}) that $a_{11} = a_{21}
= a_{31} = a_{41} = 0$, $a_{12} = a_{22} = a_{32} = a_{42} = 0$ and
$a_{13} = a_{23} = a_{33} = a_{43} = 0$, respectively. Now, a direct
calculation similar to the given above shows that
$P^l_{235}=P^l_{245}=P^l_{345}=0$ for $l=1,\ldots,4$ if and only if
all the coefficients $b_{ij}$ are zero. But then
$P^3_{124}=\lambda(a^2+b^2+c^2)\not= 0$, that is to say, the Jacobi
identity is not satisfied. This proves that the coefficient
$\lambda$ in (\ref{lambda}) vanishes.

Since $\lambda=0$ then (\ref{ces}) and (\ref{lambda}) imply that all
the coefficients $c_{ij}$ in (\ref{gen-dif}) vanish, and therefore
the Lie algebra $\mathfrak g$ is an extension of a 4-dimensional Lie
algebra $\mathfrak a$ having a triple of 2-forms $F_1,F_2,F_3$
satisfying $F_i\wedge F_i=F_j\wedge F_j\not=0$ and $F_i\wedge F_j=0$
for $i\not= j$. This gives a hyperK\"ahler structure on $\mathfrak
a$ and it follows that the Lie algebra $\mathfrak a$ is necessarily
abelian. Therefore, all the coefficients $a_{ij},b_{ij},c_{ij}$
vanish, i.e. (\ref{gen-dif}) reduces to (\ref{diff-heisenberg}) and
$\mathfrak g$ is isomorphic to ${\mathfrak h}(2,1)$.
\end{proof}

From now on, we restrict our attention to the { three parametric } family of
$SU(2)$-structures $(\eta,F_1,F_2,F_3)$ given by
(\ref{diff-heisenberg})--(\ref{structure}). Let $\Gamma(2,1)$ denote
the subgroup of matrices of $H(2,1)$ with integer entries and
consider the compact nilmanifold $N(2,1)=\Gamma(2,1)\backslash
H(2,1)$. We can describe $N(2,1)$ as a principal circle bundle over
a $4$-torus
$$S^1 \hookrightarrow N(2,1) \to T^4,$$
by the projection $(x_1,y_1,x_2,y_2,z) \mapsto (x_1,y_1,x_2,y_2)$.
Since the $SU(2)$-structure given by
(\ref{diff-heisenberg})--(\ref{structure}) is left invariant, it
descends to an $SU(2)$-structure on the compact manifold $N(2,1)$
satisfying~(\ref{solstr}). { We denote it by $(N(2,1)_{a,b,c},\eta,F_1,F_2,F_3)$}.

Since the torsion 3-form $T$ of the $SU(2)$-structure is
\begin{equation}\label{toree}
T=\eta\wedge d\eta = a\, e^{125} + b\, e^{135} + c\, e^{145} - c\,
e^{235} + b\, e^{245} - a\, e^{345},
\end{equation}
we have using \eqref{diff-heisenberg} that
\begin{equation}\label{d-torsion}
dT= -2(a^2+b^2+c^2) e^{1234}.
\end{equation}
It is straightforward to check that $T$ is parallel with respect
to the torsion connection $\nabla^+$, i.e.
\begin{lemma}\label{parallel}
For any $a,b,c\in \mathbb{R}$ such that $a^2+b^2+c^2\not=0$, we have
$\nabla^+ T=0$.
\end{lemma}
On the other hand, using \eqref{pmcon} and the expression \eqref{toree}, we calculate
the non-zero curvature forms
$(\Omega^+)^i_j=-(\Omega^+)^j_i$ of the torsion connection are
determined by:
\begin{equation}\label{curvature}
\begin{array}{l}
(\Omega^+)^1_2 = -(\Omega^+)^3_4 = -a\, de^5,\quad (\Omega^+)^1_3=
(\Omega^+)^2_4= -b\, de^5,\quad (\Omega^+)^1_4 = -(\Omega^+)^2_3=
-c\, de^5.
\end{array}
\end{equation}

Next we find a large family of $SU(2)$-instantons { depending on three (real) parameters} .

\begin{prop}\label{instanton}
Let $A_{\lambda,\mu,\tau}$ be the linear connection on $N(2,1)$
defined by the connection forms
$$
\begin{array}{l}
(\sigma^{A_{\lambda,\mu,\tau}})^1_2 = -
(\sigma^{A_{\lambda,\mu,\tau}})^2_1= -(\sigma^{A_{\lambda,\mu,\tau}})^3_4
=(\sigma^{A_{\lambda,\mu,\tau}})^4_3 = -\lambda\, e^5,\\[8pt]
(\sigma^{A_{\lambda,\mu,\tau}})^1_3 = -
(\sigma^{A_{\lambda,\mu,\tau}})^3_1= (\sigma^{A_{\lambda,\mu,\tau}})^2_4 = -
(\sigma^{A_{\lambda,\mu,\tau}})^4_2 = -\mu\, e^5,\\[8pt]
(\sigma^{A_{\lambda,\mu,\tau}})^1_4 = -
(\sigma^{A_{\lambda,\mu,\tau}})^4_1= -(\sigma^{A_{\lambda,\mu,\tau}})^2_3
=(\sigma^{A_{\lambda,\mu,\tau}})^3_2 = -\tau\, e^5,
\end{array}
$$
and $(\sigma^{A_{\lambda,\mu,\tau}})^i_j=0$ for the remaining
$(i,j)$, where $\lambda,\mu,\tau\in \mathbb{R}$. Then,
$A_{\lambda,\mu,\tau}$ is an $SU(2)$-instanton with respect to any
of the $SU(2)$-structures $(\eta,F_1,F_2,F_3)$ given by
$(\ref{diff-heisenberg})$--$(\ref{structure})$,
$A_{\lambda,\mu,\tau}$ preserves the metric, and its first
Pontrjagin form is given by
$$p_1(A_{\lambda,\mu,\tau})=-\frac{(\lambda^2+\mu^2+\tau^2)(a^2+b^2+c^2)}{2\pi^2} e^{1234}.$$
\end{prop}
\begin{proof}
A direct calculation shows that the non-zero curvature forms
$(\Omega^{A_{\lambda,\mu,\tau}})^i_j$ of the connection
$A_{\lambda,\mu,\tau}$ are:
$$
\begin{array}{l}
(\Omega^{A_{\lambda,\mu,\tau}})^1_2 =-
(\Omega^{A_{\lambda,\mu,\tau}})^2_1= -(\Omega^{A_{\lambda,\mu,\tau}})^3_4
=(\Omega^{A_{\lambda,\mu,\tau}})^4_3 = -\lambda\, de^5,\\[8pt]
(\Omega^{A_{\lambda,\mu,\tau}})^1_3 = -
(\Omega^{A_{\lambda,\mu,\tau}})^3_1 =(\Omega^{A_{\lambda,\mu,\tau}})^2_4 = -
(\Omega^{A_{\lambda,\mu,\tau}})^4_2 = -\mu\, de^5,\\[8pt]
(\Omega^{A_{\lambda,\mu,\tau}})^1_4= -
(\Omega^{A_{\lambda,\mu,\tau}})^4_1= -(\Omega^{A_{\lambda,\mu,\tau}})^2_3=
(\Omega^{A_{\lambda,\mu,\tau}})^3_2 = -\tau\, de^5.
\end{array}
$$
Hence $A_{\lambda,\mu,\tau}$ satisfies (\ref{25}); in fact, since
$F_1=e^{12}+e^{34}$, the almost complex structure $\psi$ is given by
$\psi(E_1)=-E_2$, $\psi(E_3)=-E_4$. Therefore, the connection
$A_{\lambda,\mu,\tau}$ is an $SU(2)$-instanton.
\end{proof}
The following results give explicit compact valid solutions on
$N(2,1)$ ({ depending on six real parameters)} to the heterotic supersymmetry equations with non-zero flux
and constant dilaton satisfying the anomaly cancellation condition
with respect to $\nabla^+$ and the Levi-Civita connection
$\nabla^g$ and show that all our solutions for $\nabla^+$ also solve
the equations of motion.

\begin{thrm}\label{N(2,1)}
Let $(N(2,1)_{a,b,c},\eta,F_1,F_2,F_3)$ be a compact $SU(2)$-nilmanifold as above, $\nabla^+$
the torsion connection and $A_{\lambda,\mu,\tau}$ the
$SU(2)$-instanton given in Proposition~\ref{instanton}. Let
$(\lambda,\mu,\tau)\not=(0,0,0)$ be such that
$\lambda^2+\mu^2+\tau^2 < a^2+b^2+c^2$; then
$$
dT= 2\pi^2 \alpha' \, (p_1(\nabla^+)- p_1(A_{\lambda,\mu,\tau})),
$$
where $\alpha' = 2(a^2+b^2+c^2-\lambda^2-\mu^2-\tau^2)^{-1} >0$.

Therefore, the nilmanifold  $((N(2,1)_{a,b,c},\eta,F_1,F_2,F_3),A_{\lambda,\mu,\tau},\nabla^+)$ is a compact
solution to the supersymmetry equations \eqref{sup1} obeying the anomaly cancellation
\eqref{acgen} and solving the equations of motion \eqref{mot} in dimension 5.

Denote  $r=a^2+b^2+c^2$, the Riemannian metric can be expressed  locally  by
\begin{itemize}
\item[i)]  If $(a,b)\not=(0,0)$ then
$$ g =  r^2 (dx_1)^2 + r^2 (dx_2)^2 + r (dy_1)^2
                                  + r (a^2+b^2) (dy_2)^2
                                  + r^4 (x_1 dy_1 + x_2 dy_2 - dz)^2,$$
 \item[ii)] If $a=b=0$ then
 $$ g =  (dx_1)^2 + (dx_2)^2 + (dy_1)^2 + (dy_2)^2
                                  + c^2 (x_1 dy_1 + x_2 dy_2 - dz)^2.
$$
\end{itemize}

\end{thrm}

\begin{proof}
The non-zero curvature forms of the torsion connection $\nabla^+$
are given by \eqref{curvature}, which implies that its first
Pontrjagin form is
$$
p_1(\nabla^+)=  - \frac{(a^2 + b^2 + c^2)^2}{2 \pi^2} e^{1234}.
$$
Now the proof follows directly from \eqref{d-torsion} and
Proposition~\ref{instanton}. The final assertion in the theorem
follows from Lemma~\ref{parallel} and Corollary~\ref{thacp}.
\end{proof}
\begin{prop}\label{N(2,1)-LC}
Let $(N(2,1)_{a,b,c},\eta,F_1,F_2,F_3)$ be a compact $SU(2)$-nilmanifold as above, $\nabla^g$
the Levi-Civita connection and $A_{\lambda,\mu,\tau}$ the
$SU(2)$-instanton given in Proposition~\ref{instanton}. Let
$(\lambda,\mu,\tau)\not=(0,0,0)$ be such that
$\lambda^2+\mu^2+\tau^2 < \frac38(a^2+b^2+c^2)$; then
$$
dT= 2\pi^2 \alpha' \, (p_1(\nabla^g)- p_1(A_{\lambda,\mu,\tau})),
$$
where $\alpha' = 16 \left( 3(a^2+b^2+c^2)-8(\lambda^2+\mu^2+\tau^2)
\right)^{-1}
>0$.

Therefore, the nilmanifold  $((N(2,1)_{a,b,c},\eta,F_1,F_2,F_3),A_{\lambda,\mu,\tau},\nabla^g)$ is a compact
solution to the supersymmetry equations \eqref{sup1} satisfying  the anomaly cancellation
condition \eqref{acgen}.
\end{prop}

\begin{proof}
The non-zero curvature forms $(\Omega^g)^i_j=-(\Omega^g)^j_i$ of the
Levi-Civita connection $\nabla^g$ are determined by $(\Omega^g)^i_5=\frac14 (a^2+b^2+c^2) e^{i5}$,
for $i=1,\ldots,4$, and
$$
\begin{array}{l}
(\Omega^g)^1_2 = - \frac{3a}{4} de^5 - \frac14 (a^2+b^2+c^2)
e^{34},\qquad (\Omega^g)^1_3 = - \frac{3b}{4} de^5 + \frac14
(a^2+b^2+c^2) e^{24},\\[8pt]
(\Omega^g)^1_4 = - \frac{3c}{4} de^5 - \frac14 (a^2+b^2+c^2)
e^{23},\qquad (\Omega^g)^2_3 = \frac{3c}{4} de^5 - \frac14
(a^2+b^2+c^2) e^{14},\\[8pt] (\Omega^g)^2_4 = - \frac{3b}{4} de^5 + \frac14
(a^2+b^2+c^2) e^{13},\qquad (\Omega^g)^3_4 = \frac{3a}{4} de^5 -
\frac14 (a^2+b^2+c^2) e^{12}.
\end{array}
$$
This implies that the first Pontrjagin form of $\nabla^g$ is
$$
p_1(\nabla^g)= -\frac{3}{16 \pi^2} (a^2+b^2+c^2)^2 e^{1234}.
$$
Now the proof follows directly from \eqref{d-torsion} and
Proposition~\ref{instanton}.
\end{proof}

\begin{rmrk}
The first Pontrjagin form of the connection $\nabla^-$ is zero,
therefore there is no compact solution to the heterotic supersymmetry
equations satisfying the anomaly cancellation condition  with $\nabla=\nabla^-$.
\end{rmrk}

\section{Conclusions}

We have constructed new explicit compact supersymmetric solutions with non-zero NS 3-form
field strength, non-flat instanton and constant dilaton to the
heterotic string equations with non-trivial bosonic fields  on a five dimensional manifold.
The solutions are compact 5-dimensional nilmanifolds $N(2,1)_{a,b,c}$ which are $S^1$-bundles over
a four torus $T^4$ equipped with anti-self-dual connection { whose curvature depends on three real constants $a,b,c$,}
satisfying the heterotic supersymmetry
equations \eqref{sup1} with non-zero NS 3-form $H$, non-trivial instanton $A_{\lambda,\mu,\tau}$
depending on { another} three { real parameters}  $\lambda,\mu, \tau$, and a constant dilaton obeying the
three-form Bianchi identity  \eqref{acgen} with curvature term taken with respect to either  $R=R^+$ or $R=R^g$.

We have analyzed whether the
heterotic supersymmetry equations together with the three form Bianchi identity with non-trivial curvature
term $TrR\wedge R$ imply the equations of motion (the Einstein equation of motion)
in dimension $d=5$. We have found a quadratic condition on the curvature $R$ which is necessary
and sufficient for the heterotic supersymmetry and the anomaly cancellation to imply
 the heterotic equations of motion \eqref{mot} in dimension five. Based on that we
achieved the heterotic equations of motion are satisfied for curvature term taken with respect to $R^+$ showing
that $R^+$ is an $SU(2)$-instanton on $N(2,1)_{a,b,c}$ and therefore the quadratic curvature condition is fulfilled.

{ Our { six-parametric} solutions can be viewed as examples of a (good) half-symmetric solutions to the heterotic supergravity, i.e.
heterotic solutions with 8 supersymmetries preserved, for which $Hol(\nabla^+)\subseteq
SU(2)$ according to the general scheme developed in \cite{Pap},
(see also \cite{GLP,GPRS}). Since $N(2,1)_{a,b,c}$ is compact, it might  be considered as the vacuum
of compactifications with fluxes to 5 dimensions.}

We have given structure equations of any solution to the first two Killing spinor equations
in \eqref{sup1}  in terms of exterior derivatives of an
$SU(2)$-structure in dimension five, a notion introduced in
\cite{ConS,GGMPR}, and express its Ricci tensor in terms of the
structure's forms. Mathematically these structures were known as quasi-Sasaki not Sasaki
5-manifolds with anti-self-dual exterior derivative of the almost contact form
and their special conformal transformations. In the case when the
quasi-Sasaki structure is regular the solutions to the first two
equations in \eqref{sup1} are $S^1$-bundles over a Calabi-Yau
4-manifold with anti-self-dual curvature 2-form. Any solution with non-constant dilaton
 arises from a solution with constant dilaton via special conformal transformations and
 the dilaton depends only on the Calabi-Yau base (cf also \cite{FI,FI2,Pap}).

An open problem is whether the more general case of solution can occur. Namely, does there exist a  compact non-regular (the
integral curves of the Reeb vector fielf $\xi$ are not closed)  quasi-Sasaki 5-manifold
with anti-self-dual 2-form $d\eta$, or more precisely,  non-regular $SU(2)$-
structures obeying \eqref{solstr} (\eqref{solstr1} in the case of non-constant dilaton) on a compact 5-manifold?
We do not know any examples of this kind. Such examples might  also be relevant in
constructing compact heterotic solutions  in dimension six since the
construction of $\mathbb T^2$-bundles over Calabi-Yau surface
presented in \cite{GP} can be generalized to a circle bundle over
such an example solving automatically the first two equations in
\eqref{sup1} (see e.g. \cite{FIUV1}).

\medskip
\noindent {\bf Acknowledgments.}  We would like to thank the referee for his valuable comments and remarks.
This work has been partially
supported through grants MEC (Spain) MTM2005-08757-C04-02,
MTM2008-06540-C02/01-02 and under project ``Ingenio Mathematica
(i-MATH)" No. CSD2006-00032 (Consolider Ð Ingenio 2010). S.I. is
partially supported by the Contract 082/2009 with the University
of Sofia `St.Kl.Ohridski' and Contract ``Idei", DO
02-257/18.12.2008. S.I. is a Senior Associate to the Abdus Salam
ICTP, Trieste and the final stage of the research was done during
his stay in the ICTP, Fall 2008.

\end{document}